\documentclass[a4paper,12pt,twoside,leqno,final]{amsart}
\usepackage{amsmath}
\usepackage{amssymb}
\usepackage{hyperref}

\newtheorem{thm}{Theorem}[section]
\newtheorem{lem}[thm]{Lemma}

\newtheorem{prop}[thm]{Proposition}
\newtheorem{cor}[thm]{Corollary}

\theoremstyle{definition}

\newcommand{\C}{{\mathbb C}}
\newcommand{\D}{{\mathbb D}}

\newcommand{\T}{{\mathbb T}}

\newcommand{\N}{{\mathbb N}}

\newcommand{\f}{\frac}
\newcommand{\ov}{\overline}

\newcommand{\de}{\delta}

\newcommand{\ze}{\zeta}
\renewcommand{\th}{\theta}
\newcommand{\si}{\sigma}

\numberwithin{equation}{section}

\title[Carleson-type embeddings with closed range]
{Carleson-type embeddings with closed range}
\author{Konstantin M. Dyakonov}
\address{Departament de Matem\`atiques i Inform\`atica, IMUB, BGSMath, Universitat de Barcelona, Gran Via 585, E-08007 Barcelona, Spain}
\address{ICREA, Pg. Llu\'is Companys 23, E-08010 Barcelona, Spain}
\email{konstantin.dyakonov@icrea.cat}
\thanks{Supported in part by grants PID2021-123405NB-I00 and PID2024-160033NB-I00 from El Ministerio de Ciencia, Innovaci\'on y Universidades (Spain).}

\begin{document}
\begin{abstract}
We characterize the Carleson measures $\mu$ on the unit disk for which the image of the Hardy space $H^p$ under the corresponding embedding operator is closed in $L^p(\mu)$. In fact, a more general result involving $(p,q)$-Carleson measures is obtained. A similar problem is solved in the setting of Bergman spaces.
\end{abstract}

\maketitle

\section{Introduction and statement of results}

Let $\D$ stand for the disk $\{z\in\C:\,|z|<1\}$, $\T$ for its boundary, and $m$ for the normalized arc length measure on $\T$. The Lebesgue space $L^p:=L^p(\T,m)$ with $0<p\le\infty$ is then defined in the usual manner and endowed with the standard norm $\|\cdot\|_p$\,; when $0<p<1$, the term \lq\lq quasinorm" should actually be used. The functions in $L^p$ are assumed to be complex-valued.

\par Given a holomorphic function $f$ on $\D$, one says that $f$ is in the {\it Hardy space} $H^p$ (with a finite $p>0$) if 
$$\|f\|_{H^p}:=\sup_{0<r<1}\left(\int_\T|f(r\ze)|^p\,dm(\ze)\right)^{1/p}<\infty,$$
and that $f$ is in $H^\infty$ if 
$$\|f\|_{H^\infty}:=\sup_{z\in\D}|f(z)|<\infty.$$
As usual, we identify each function $f\in H^p$ with its boundary trace on $\T$; the latter, denoted by $f$ again, is defined almost everywhere on $\T$ by means of radial---or nontangential---limits. Moreover, this boundary function is in $L^p$ and $\|f\|_p=\|f\|_{H^p}$. See, e.g., \cite[Chapter II]{G} for these matters and other basic facts about $H^p$.

\par Now suppose that $\mu$ is a (positive) Borel measure on $\D$, and let $0<p<\infty$. Write $L^p(\mu):=L^p(\D,\mu)$ for the corresponding Lebesgue space, i.e., the space of Borel measurable functions $f:\D\to\C$ with
$$\|f\|_{L^p(\mu)}:=\left(\int_{\D}|f|^p\,d\mu\right)^{1/p}<\infty.$$
Recall that $\mu$ is said to be a {\it Carleson measure} if for some (and then every) $p\in(0,\infty)$ one has $H^p\subset L^p(\mu)$. By a classical result (see \cite[Chapter II]{G} or \cite[Lecture VII]{NShift}), $\mu$ is a Carleson measure if and only if
$$\sup_I\f{\mu(Q_I)}{m(I)}<\infty,$$
where $I$ ranges over the open arcs on $\T$ and $Q_I$ stands for the \lq\lq Carleson square" with base $I$; that is, 
$$Q_I:=\left\{z\in\D:\,z/|z|\in I,\,\,|z|\ge1-m(I)\right\}$$
(when $z=0$, we put $z/|z|=1$). This characterization is known as Carleson's embedding theorem. 

\par Given a Carleson measure $\mu$ and a number $p\in(0,\infty)$, we write $j_p(\mu)$ for the natural embedding (or inclusion) operator that arises, going from $H^p$ to $L^p(\mu)$, whose action on a function from $H^p$ consists in restricting it to the support of $\mu$. This map
\begin{equation}\label{eqn:embpphar}
j_p(\mu):\,H^p\hookrightarrow L^p(\mu)
\end{equation}
is necessarily continuous (by the closed graph theorem), and it may happen to have various other properties; the corresponding criteria should be expressible in terms of $\mu$. The question that interests us here is: {\it For which Carleson measures $\mu$ does the embedding operator \eqref{eqn:embpphar} have closed range in $L^p(\mu)$?} Such measures $\mu$ are then said to have the closed range property.

\par One trivial example is provided by any measure $\mu$ supported on a finite set of points in $\D$. In this case, $L^p(\mu)$ is finite dimensional, while the operator \eqref{eqn:embpphar} is clearly surjective and hence has closed range. 

\par To see a more interesting example, consider an infinite sequence $\mathcal Z=\{z_n\}$ of pairwise distinct points in $\D$. Assume, in addition, that $\mathcal Z$ is an {\it interpolating} (or {\it $H^\infty$-interpolating}) {\it sequence}, meaning that 
$$H^\infty\big|_{\mathcal Z}=\ell^\infty.$$
(Here and below, given a function space $X$ on $\D$ and a sequence $\mathcal Z=\{z_n\}\subset\D$, the {\it trace space} $X\big|_{\mathcal Z}$ is defined as the set of all sequences $\{f(z_n)\}$ that arise as $f$ ranges over $X$.) By a famous theorem of Carleson (see, e.g., \cite[Chapter VII]{G}), $\mathcal Z$ is an interpolating sequence if and only if 
\begin{equation}\label{eqn:carconint}
\inf_k\prod_{j:\,j\ne k}\left|\f{z_j-z_k}{1-\ov z_jz_k}\right|>0.
\end{equation}
A well-known consequence of \eqref{eqn:carconint} is that the measure 
$$\mu_{\mathcal Z}:=\sum_n(1-|z_n|)\,\de_{z_n}$$
(where $\de_{z_n}$ stands for the unit point mass at $z_n$) is Carleson, i.e., 
$$H^p\big|_{\mathcal Z}\subset L^p(\mu_{\mathcal Z})$$
for some/any finite $p>0$. To keep on the safe side, we remark that $L^p(\mu_{\mathcal Z})$ is also viewed as a sequence space, its elements being precisely the sequences $\{w_n\}$ with $\sum_n|w_n|^p(1-|z_n|)<\infty$. In fact, the condition that $\mu_{\mathcal Z}$ be a Carleson measure (sometimes referred to as the {\it Carleson--Newman condition}) characterizes the sequences $\mathcal Z$ that are finite unions of interpolating sequences; see, e.g., \cite[p.\,109]{DS} or \cite[p.\,158]{NShift}. 

\par Now, if $\mathcal Z$ is an interpolating sequence, then we actually have 
\begin{equation}\label{eqn:shshkaint}
H^p\big|_{\mathcal Z}=L^p(\mu_{\mathcal Z})
\end{equation}
for every $p\in(0,\infty)$. Conversely, \eqref{eqn:shshkaint} implies that $\mathcal Z$ is an interpolating sequence. These facts were established by Shapiro and Shields \cite{SS} for $1\le p<\infty$, and then extended by Kabaila \cite{K} to the case $0<p<1$. 

\par Thus, if $\mathcal Z$ is an interpolating sequence and $\mu_{\mathcal Z}$ is the associated Carleson measure, then the embedding operator $j_p(\mu_{\mathcal Z})$---which is no other than the restriction map $f\mapsto f\big|_{\mathcal Z}$ on $H^p$---has all of $L^p(\mu_{\mathcal Z})$ for its range; in particular, the range is closed. Clearly, a similar conclusion holds for any \lq\lq perturbation" $\mu$ of $\mu_{\mathcal Z}$ given by 
$$\mu=\sum_na_n\de_{z_n},$$
where $a_n$ are positive numbers with $a_n\asymp1-|z_n|$, meaning that 
$$0<\inf_n\f{a_n}{1-|z_n|}\le\sup_n\f{a_n}{1-|z_n|}<\infty.$$
(In general, we use the sign $\asymp$ to mean that the ratio of the two quantities takes values in an interval of the form $[C^{-1},C]$ with some fixed constant $C\ge1$.) Indeed, for such a perturbed measure $\mu$, we obviously have $L^p(\mu)=L^p(\mu_{\mathcal Z})$. 

\par Having mentioned these examples of Carleson measures with the closed range property, we now claim that no other examples actually exist. This is implied by one of our results, namely Theorem \ref{thm:hardy}, to be stated and proved below.

\par A more general version of the problem arises if we fix two positive numbers, $p$ and $q$, and assume that $\mu$ is a {\it $(p,q)$-Carleson measure}. This means, by definition, that $\mu$ is a positive Borel measure on $\D$ satisfying $H^p\subset L^q(\mu)$. Such measures have been described geometrically in terms of Carleson squares (in the spirit of Carleson's original theorem, where $p=q$); this was done in \cite{Du} for $p<q$ and in \cite{V} for $p>q$. The latter case was also treated in \cite{LPLMS}. For a comprehensive survey of these results, the reader is referred to \cite{BJ}. 

\par Given a $(p,q)$-Carleson measure $\mu$, we write $j_{p,q}(\mu)$ for the associated (necessarily continuous) embedding---or inclusion---operator going from $H^p$ to $L^q(\mu)$. (The symbol $j_p(\mu)$ used above is thus an abbreviation for $j_{p,p}(\mu)$.) We then ask the same question as before, or rather a $(p,q)$-version thereof. Namely, we want to know whether the embedding map
\begin{equation}\label{eqn:embpqhar}
j_{p,q}(\mu):\,H^p\hookrightarrow L^q(\mu)
\end{equation}
has closed range. This certainly happens if $\mu$ is supported on a finite set of points in $\D$, so we may safely exclude this trivial case from further consideration. Otherwise, the answer is given by the following result. 

\begin{thm}\label{thm:hardy} Let $0<p,q<\infty$ and let $\mu$ be a $(p,q)$-Carleson measure supported on an infinite subset of $\D$. The following are equivalent: 
\par{\rm(i)} The range of the embedding operator \eqref{eqn:embpqhar} is closed in $L^q(\mu)$;
\par{\rm(ii)} $p=q$ and $\mu$ has the form $\sum_na_n\de_{z_n}$, where $\{z_n\}$ is an interpolating sequence in $\D$ and the positive numbers $a_n$ satisfy $a_n\asymp1-|z_n|$.
\end{thm}

\par As an immediate consequence, we obtain an amusing characterization of interpolating sequences. 

\begin{cor}\label{cor:intseq} Let $0<p<\infty$ and let $\mathcal Z=\{z_n\}$ be a sequence of pairwise distinct points in $\D$. Then $\mathcal Z$ is an interpolating sequence if and only if one can find a Carleson measure $\mu$, supported precisely on $\mathcal Z$, such that the embedding operator $j_p(\mu):H^p\hookrightarrow L^p(\mu)$ has closed range.
\end{cor}

\par Here, by saying that $\mu$ is supported precisely on $\mathcal Z$ we mean that $\mu$ lives on $\mathcal Z$ and assigns positive mass to every point therein.

\par We now move on to consider a similar problem in the Bergman space setting. Let $\si$ denote normalized area measure on $\D$. For $0<p<\infty$, the {\it Bergman space} $A^p$ is defined as the set of all holomorphic functions $f$ on $\D$ that are in $L^p(\si)$; one puts then
$$\|f\|_{A^p}:=\|f\|_{L^p(\si)}.$$
A Borel measure $\mu$ on $\D$ is said to be a {\it Bergman-Carleson measure} if for some (or every) $p\in(0,\infty)$ one has $A^p\subset L^p(\mu)$. This last inclusion is, of course, equivalent to saying that 
\begin{equation}\label{eqn:bergcarl}
\|f\|_{L^p(\mu)}\le C\|f\|_{A^p},\qquad f\in A^p,
\end{equation}
for some constant $C=C(p,\mu)$. It was shown in \cite{Has} that the class of measures $\mu$ making \eqref{eqn:bergcarl} true is indeed independent of $p$, $0<p<\infty$, and that $\mu$ has this property if and only if
$$\sup_I\f{\mu(Q_I)}{m(I)^2}<\infty,$$
where $I$ ranges over the open arcs on $\T$. Alternative characterizations of Bergman-Carleson measures can be found in \cite[Chapter 2]{DS}. 

\par More generally, given $0<p,q<\infty$ and a Borel measure $\mu$ on $\mathbb D$, we say that $\mu$ is a {\it $(p,q)$-Bergman-Carleson measure} if $A^p\subset L^q(\mu)$. With any such $\mu$ we associate the (bounded) embedding operator 
\begin{equation}\label{eqn:capjpq}
J_{p,q}(\mu):A^p\hookrightarrow L^q(\mu),
\end{equation}
whose action on a function $f\in A^p$ amounts to restricting $f$ to the support of $\mu$. When $p=q$, a \lq\lq $(p,p)$-Bergman-Carleson measure" is just a Bergman-Carleson measure, and we write $J_p(\mu)$ for $J_{p,p}(\mu)$. 

\par In this context, our question takes the following form: {\it For which $(p,q)$-Bergman-Carleson measures $\mu$ is the range of the embedding operator \eqref{eqn:capjpq} closed in $L^q(\mu)$?} To state the answer, we need to introduce some more terminology.

\par Given a Bergman-Carleson measure $\mu$ and a number $p\in(0,\infty)$, we say that $\mu$ is a {\it sampling measure for $A^p$} (or an {\it $A^p$-sampling measure}) if there is a constant $c>0$ such that
\begin{equation}\label{eqn:sampmeas}
\|f\|_{L^p(\mu)}\ge c\|f\|_{A^p}
\end{equation}
for all $f\in A^p$. Clearly, we may couple the \lq\lq reverse Carleson estimate" \eqref{eqn:sampmeas} with the forward inequality \eqref{eqn:bergcarl} to get 
\begin{equation}\label{eqn:twosides}
\|f\|_{L^p(\mu)}\asymp\|f\|_{A^p},\qquad f\in A^p,
\end{equation}
so our $A^p$-sampling measures $\mu$ are precisely those which make \eqref{eqn:twosides} true (and this class of measures does depend on $p$). An obvious example of an $A^p$-sampling measure is, of course, provided by $\si$. As a more interesting example, we mention the measure $|B|^p\,d\si$, where $B$ is a Blaschke product whose zero-set is a finite union of interpolating sequences; see Theorem 10 in \cite[Chapter 4]{DS}.

\par A sequence $\mathcal Z=\{z_n\}$ of pairwise distinct points in $\D$ is, by definition, a {\it sampling sequence for $A^p$} (or an {\it $A^p$-sampling sequence}) if the associated measure 
\begin{equation}\label{eqn:nuznuz}
\nu_{\mathcal Z}:=\sum_n(1-|z_n|)^2\de_{z_n}
\end{equation}
is a sampling measure for $A^p$. Such sequences were characterized by Seip (along with $A^p$-interpolating sequences, to be defined shortly) in terms of densities; see \cite{SeInv, Se} or \cite[Chapter 6]{DS}. As to general sampling measures, these were subsequently described by Luecking in \cite{LMA}. (A minor terminological discrepancy should, however, be pointed out: a measure $\mu$ is $A^p$-sampling in our current sense if and only if $(1-|z|)^{-2}d\mu(z)$ is $A^p$-sampling in the sense of \cite{LMA}.)

\par Finally, we recall that a sequence of pairwise distinct points $\mathcal Z=\{z_n\}$ in the disk is said to be an {\it $A^p$-interpolating sequence} if the trace space $A^p|_{\mathcal Z}$ coincides with the sequence space 
$$\left\{\{w_n\}:\,\sum_n|w_n|^p(1-|z_n|)^2<\infty\right\},$$
which in turn can be naturally identified with $L^p(\nu_{\mathcal Z})$. The reader is again referred to \cite{SeInv, Se} or \cite[Chapter 6]{DS} for a characterization.

\begin{thm}\label{thm:bergman} Let $0<p,q<\infty$ and let $\mu$ be a $(p,q)$-Bergman-Carleson measure supported on an infinite subset of $\D$. The following are equivalent: 
\par{\rm(i)} The range of the embedding operator \eqref{eqn:capjpq} is closed in $L^q(\mu)$;
\par{\rm(ii)} $p=q$ and one of the two conditions holds: either $\mu$ is a sampling measure for $A^p$, or $\mu$ has the form $\sum_na_n\de_{z_n}$, where $\{z_n\}$ is an $A^p$-interpolating sequence in $\D$ and the positive numbers $a_n$ satisfy $a_n\asymp(1-|z_n|)^2$.
\end{thm}

\par Here, again, the implication (ii)$\implies$(i) is immediate. Indeed, if $\mu$ is a sampling measure for $A^p$, then \eqref{eqn:sampmeas} tells us that the embedding operator $J_p(\mu):A^p\hookrightarrow L^p(\mu)$ is bounded below and therefore has closed range. Now assume that $\mathcal Z=\{z_n\}$ is an $A^p$-interpolating sequence and $\mu=\sum_na_n\de_{z_n}$, with $a_n\asymp(1-|z_n|)^2$. Consider also the measure $\nu_{\mathcal Z}$ given by \eqref{eqn:nuznuz}. The range of $J_p(\mu)$ is then the trace space $A^p|_{\mathcal Z}$, so it coincides with $L^p(\nu_{\mathcal Z})$ or, equivalently, with $L^p(\mu)$. In particular, the range of $J_p(\mu)$ is closed.

\par Theorem \ref{thm:bergman} can certainly be extended to weighted Bergman spaces, at least for standard power-type weights, but we do not pursue this line here. Rather, we want to compare the Hardy and Bergman versions of the problem, highlighting the subtleties that arise in the latter case. We also mention that some partial results related to Theorem \ref{thm:bergman} can be found in \cite{LPac}. 

\par To state our last result, which is an amusing consequence of Theorem \ref{thm:bergman}, we need to introduce yet another bit of notation. Given a Carleson measure $\mu$, we let $\mathcal C_H(\mu)$ denote the set of those numbers $p\in(0,\infty)$ for which the embedding operator $j_p(\mu):H^p\hookrightarrow L^p(\mu)$ has closed range. Similarly, if $\mu$ is a Bergman-Carleson measure, we write $\mathcal C_B(\mu)$ for the set of those $p\in(0,\infty)$ for which the embedding operator $J_p(\mu):A^p\hookrightarrow L^p(\mu)$ has closed range. (Of course, the subscripts in $\mathcal C_H$ and $\mathcal C_B$ stand for \lq\lq Hardy" and \lq\lq Bergman", respectively.)

\par It follows from Theorem \ref{thm:hardy} that, for any Carleson measure $\mu$, we have either $\mathcal C_H(\mu)=\emptyset$ or $\mathcal C_H(\mu)=(0,\infty)$. Things become different in the Bergman setting, though, in which case $\mathcal C_B(\mu)$ may well be disconnected.

\begin{prop}\label{prop:discon} For each $p_0\in(0,\infty)$ there exists a Bergman-Carleson measure $\mu$ such that
$$\mathcal C_B(\mu)=(0,\infty)\setminus\{p_0\}.$$
\end{prop}

\par We remark, by way of digression, that a similar study of Carleson measures with the closed range property could definitely be performed for many other function spaces. In particular, various subspaces of $H^p$ and $A^p$ are worth looking at. To mention one specific case of interest---and a possible topic of future research---consider the {\it star-invariant} (or {\it model}) {\it subspace} $K^p_\th:=H^p\cap\th\ov z\ov{H^p}$, with $1\le p<\infty$, associated with an inner function $\th$. In general, Carleson measures for $K^p_\th$ are not completely well understood, and the problem of describing them (first posed in \cite{Co}) has a long history. Among the many papers that treat it, we cite \cite{A1, DAJM, TV} and \cite[pp.\,80--81]{L}. It is conceivable, though, that if the closed range condition is imposed, a neater characterization might come out. Here, we limit ourselves to observing that, while defined in the $H^p$ context, star-invariant subspaces tend to exhibit certain Bergman-type features. For instance, they possess nontrivial sampling measures (see, e.g., \cite[Theorem 3]{DLOMI} for a class of examples and \cite{BFGHR} for further details), their zero-sets vary with $p$ (see \cite{DJAM}), etc. Thus, with regard to the problem at hand, one would expect the situation with $K^p_\th$ to follow the $A^p$ pattern rather than that of $H^p$.

\par Going back to the original question about the range of the embedding map \eqref{eqn:embpphar}, one may wonder what happens if the Carleson measure $\mu$ is supported on the closed disk $\ov\D:=\D\cup\T$, not just on $\D$. Assuming that $\mu$ is a Borel measure on $\ov\D$ whose boundary component $\mu_\T:=\mu|_{\T}$ is absolutely continuous with respect to $m$, we still say that $\mu$ is Carleson if $H^p\subset L^p(\mu)$ for some, or any, finite $p>0$. For this to happen, the interior component $\mu_\D:=\mu|_{\D}$ must be a (classical) Carleson measure, while $\mu_\T$ must be of the form $d\mu_\T=h\,dm$ for some nonnegative function $h\in L^\infty$; see, e.g., \cite[Theorem 2.6]{BJ}. Now, if $\mu_\T$ (or, equivalently, $h$) is non-null, then the embedding operator $j_p(\mu)$ is one-to-one, and our question is answered by \cite[Theorem 5.3]{LLQR}. When coupled with the remark that follows it, the theorem tells us that $j_p(\mu)$ has closed range if and only if $h(=d\mu_\T/dm)$ is essentially bounded away from zero. This hopefully explains why we restrict our attention to classical Carleson measures, which live on $\D$.

\par The rest of the paper is devoted to the proofs of our current results on $H^p$ and $A^p$. In the next section, we collect a few auxiliary facts to lean upon. We then use them in the last two sections to prove Theorems \ref{thm:hardy} and \ref{thm:bergman}, along with Proposition \ref{prop:discon}.

\par Finally, the referee's helpful comments are gratefully acknowledged.

\section{Preliminaries}

Several lemmas will be needed. 

\begin{lem}\label{lem:norevcm} If $0<p<\infty$ and $\mu$ is a finite Borel measure on $\D$, then
\begin{equation}\label{eqn:nonrev}
\inf\left\{\|f\|_{L^p(\mu)}:\,f\in H^p,\,\|f\|_p=1\right\}=0.
\end{equation}
\end{lem}

\begin{proof} If the infimum in \eqref{eqn:nonrev} were positive, we would have 
\begin{equation}\label{eqn:revcarlmeas}
\|f\|_{L^p(\mu)}\ge c\|f\|_p,\qquad f\in H^p,
\end{equation}
with some fixed $c>0$. Thus, we need to show that this \lq\lq reverse Carleson estimate" is actually false.

\par When $1<p<\infty$, we can readily apply a result from \cite{HMNO}. In fact, for $p$ in this range, it follows from \cite[Theorem 2.1]{HMNO} that any measure $\mu$ on the {\it closed} disk $\ov\D$ which makes \eqref{eqn:revcarlmeas} true, with the disk algebra $H^\infty\cap C(\T)$ in place of $H^p$, must assign nonzero mass to $\T$. (Moreover, such \lq\lq reverse Carleson measures" $\mu$ on $\ov\D$ are characterized in \cite{HMNO} by the condition that the Radon-Nikodym derivative of $\mu|_\T$ with respect to $m$ is bounded away from zero. For measures that are also assumed to be Carleson, this condition appeared earlier in \cite{LLQR}.) Since our $\mu$ lives on the open disk $\D$, we see that \eqref{eqn:revcarlmeas} is sure to fail. 

\par It remains to refute \eqref{eqn:revcarlmeas} in the case where $0<p\le1$. Given such a $p$, let $N$ be a positive integer for which $Np>1$. Now put $s:=Np$ and take an arbitrary function $g\in H^s$. We have then $g^N\in H^p$, and applying \eqref{eqn:revcarlmeas} with $f=g^N$ yields
$$\|g\|_{L^s(\mu)}\ge c^{1/N}\|g\|_s,\qquad g\in H^s,$$
an estimate already disproved above.
\end{proof}

\par To keep on the safe side, let us recall the notation $\ell^p:=\ell^p(\N)$ for the space of all $p$th power summable sequences of complex numbers, normed in the natural way.

\begin{lem}\label{lem:isomembed} Let $\mu$ be a finite Borel measure on $\D$. If $2<p,q<\infty$ and $p\ne q$, then $L^q(\mu)$ does not have any subspace isomorphic to $L^p$ or $\ell^p$. 
\end{lem}

\begin{proof} It is a classical result that, for $p$ and $q$ as above, $L^q$ (the standard $L^q$ space over $\T$ or, equivalently, over the interval $[0,1]$) has no subspace isomorphic to $L^p$ or $\ell^p$; see \cite[Corollary 6.4.10]{AK} for a more general statement. 
	
\par Now, if $\mu$ is a finite sum of point masses, then $L^q(\mu)$ is finite dimensional and the lemma is trivially true. Otherwise, being infinite dimensional and separable, $L^q(\mu)$ must be isomorphic to either $\ell^q$ or $L^q$, depending on whether or not $\mu$ is purely atomic. (The separability of $L^q(\mu)$ is implied by the fact that the Borel $\sigma$-algebra of $\D$ is countably generated, and the isomorphism result just mentioned can be found in \cite[Chapter 5]{Sch} upon combining Theorem 5.1.1 with Proposition 5.4.5 therein. Alternatively, see \cite[p.\,83]{W}.) It follows that, in any event, $L^q(\mu)$ embeds isomorphically into $L^q$ (because $\ell^q$ does). Consequently, if $L^p$ or $\ell^p$ were isomorphic to a subspace of $L^q(\mu)$, it would also be isomorphic to a subspace of $L^q$. That would, however, contradict the classical result stated at the beginning of the proof.
\end{proof}

\par By way of preparation for our last lemma, suppose that $X$ is a Banach (or quasi-Banach) space of holomorphic functions on $\D$ and that $\mathcal Z=\{z_n\}$ is a sequence of pairwise distinct points in $\D$. For a sequence $\mathcal W=\{w_n\}$ from the trace space $X|_{\mathcal Z}$, its (quasi)norm therein is taken to be 
$$\inf\left\{\|f\|_X:\,f|_{\mathcal Z}=\mathcal W\right\}.$$
Further, given a (quasi-)Banach sequence space $Y$ and a sequence of complex numbers $\{m_n\}$, we say that $\{m_n\}$ is a multiplier for $Y$ if the operator $\{w_n\}\mapsto\{m_nw_n\}$ acts boundedly on $Y$. We write $\mathfrak M(Y)$ for the space of such multipliers, normed appropriately. Finally, $Y$ is called an {\it ideal space} if the (continuous) embedding $\ell^\infty\subset\mathfrak M(Y)$ holds.
 
\begin{lem}\label{lem:idealtrsp} Let $0<p<\infty$ and let $\mathcal Z=\{z_n\}$ be a sequence of pairwise distinct points in $\D$. Then 
\par{\rm (a)} $H^p|_{\mathcal Z}$ is an ideal space if and only if $\mathcal Z$ is an interpolating sequence; 
\par{\rm (b)} $A^p|_{\mathcal Z}$ is an ideal space if and only if $\mathcal Z$ is an $A^p$-interpolating sequence.
\end{lem}

\par The \lq\lq if" parts of (a) and (b) are trivial, since any weighted $\ell^p$-space is stable under multiplication by bounded sequences, while the \lq\lq only if" parts are known. In fact, for $p\ge1$, statement (a) is established in \cite[p.\,188]{NShift}; the proof given there actually works for $0<p<1$ as well. In this connection, we also cite \cite{Ha}, where a generalization of (a) is proved in the setting of Hardy-Orlicz spaces. To verify (b), one may proceed in a similar fashion, with a minor adjustment based on Schuster and Seip's Carleson-type characterization of $A^p$-interpolating sequences (see \cite{ScSe}).

\section{Proof of Theorem \ref{thm:hardy}}

We only need to prove that (i) implies (ii). As a first step, we shall deduce from (i) that $\mu$ is supported on a Blaschke set (i.e., on a subset $\mathcal Z$ of $\D$ satisfying $\sum_{z\in\mathcal Z}(1-|z|)<\infty$). Suppose this is not the case; then every $H^p$ function that vanishes $\mu$-a.e. is null, so the kernel of the embedding operator $j:H^p\hookrightarrow L^q(\mu)$ is trivial. (Here and below, we write $j$ for $j_{p,q}(\mu)$.) The range of this operator being closed in $L^q(\mu)$, we infer that $j$ maps $H^p$ isomorphically onto $j(H^p)$, so that
\begin{equation}\label{eqn:isompq}
c\|f\|_p\le\|f\|_{L^q(\mu)}\le C\|f\|_p,\qquad f\in H^p,
\end{equation}
with some constants $C>c>0$. In fact, when $p$ and $q$ are in $[1,\infty)$, the left-hand inequality in \eqref{eqn:isompq} is ensured by the standard open mapping theorem for Banach spaces, whereas a more general version involving $F$-spaces (see \cite[Chapter 2]{R}) is needed to cover the remaining cases.

\par Now let $N$ be a positive integer such that the numbers $\widetilde p:=Np$ and $\widetilde q:=Nq$ satisfy $\min(\widetilde p,\widetilde q)>2$. For any $g\in H^{\widetilde p}$ we have $g^N\in H^p$, and an application of \eqref{eqn:isompq} with $f=g^N$ gives
\begin{equation}\label{eqn:isompqtilde}
c^{1/N}\|g\|_{\widetilde p}\le\|g\|_{L^{\widetilde q}(\mu)}\le C^{1/N}\|g\|_{\widetilde p},\qquad g\in H^{\widetilde p}.
\end{equation}
This means that $H^{\widetilde p}$ embeds into $L^{\widetilde q}(\mu)$ and is isomorphic to its image (under the corresponding embedding map), which is a closed subspace of $L^{\widetilde q}(\mu)$. Since $H^{\widetilde p}$ is also isomorphic to $L^{\widetilde p}$ (see \cite{B}), it follows that $L^{\widetilde q}(\mu)$ contains an isomorphic copy of $L^{\widetilde p}$. In view of Lemma \ref{lem:isomembed} (applied with $\widetilde p$ and $\widetilde q$ in place of $p$ and $q$), this is only possible if $\widetilde p=\widetilde q$, or equivalently, if $p=q$. On the other hand, if $p=q$ then the left-hand inequality in \eqref{eqn:isompq} reduces to \eqref{eqn:revcarlmeas} and contradicts Lemma \ref{lem:norevcm}. 

\par Our claim that $\mu$ lives on a Blaschke set, say $\mathcal Z$, is now established. We also have $\#\mathcal Z=\infty$ by assumption, and we write $\mathcal Z=\{z_n\}$, where the $z_n$ are pairwise distinct points in $\D$ satisfying $\sum_n(1-|z_n|)<\infty$. The measure $\mu$ is then given by 
\begin{equation}\label{eqn:musum}
\mu=\sum_na_n\de_{z_n},
\end{equation}
the weights $a_n$ being positive numbers with $\sum_na_n<\infty$, so that $L^q(\mu)$ reduces to the appropriately weighted $\ell^q$-space. Namely, we can view functions from $L^q(\mu)$ as sequences $\mathcal W=\{w_n\}$ of complex numbers with
\begin{equation}\label{eqn:sumanwnq}
\left\|\mathcal W\right\|^q_{L^q(\mu)}=\sum_na_n|w_n|^q<\infty.
\end{equation}
Furthermore, letting $\mathcal F$ stand for the set of finitely supported sequences (i.e., those that have only finitely many nonzero elements), we note that $\mathcal F$ is dense in $L^q(\mu)$. It is also true that $\mathcal F\subset j(H^p)$, since every $\mathcal W$ from $\mathcal F$ is writable as $h|_{\mathcal Z}$ for some $h\in H^p$. Indeed, to interpolate a given sequence 
$$(w_1,\dots,w_N,0,0,\dots)\in\mathcal F$$
in this way, it suffices to put $h=B_NQ$, where $B_N$ is the Blaschke product with zeros $\{z_n\}_{n>N}$ and $Q$ is a polynomial that satisfies 
$$Q(z_n)=\f{w_n}{B_N(z_n)}\quad\text{\rm for}\quad n=1,\dots,N.$$
When coupled with condition (i), these observations allow us to conclude that $j(H^p)=L^q(\mu)$, or equivalently, 
\begin{equation}\label{eqn:tracepq}
H^p|_{\mathcal Z}=L^q(\mu).
\end{equation}

\par Finally, since $L^q(\mu)$ is obviously an ideal space, we may use \eqref{eqn:tracepq} in conjunction with Lemma \ref{lem:idealtrsp}, part (a), to infer that $\mathcal Z$ is an interpolating sequence. This in turn implies that 
\begin{equation}\label{eqn:tracepp}
H^p|_{\mathcal Z}=L^p(\mu_{\mathcal Z}), 
\end{equation}
where 
$$\mu_{\mathcal Z}:=\sum_n(1-|z_n|)\de_{z_n}.$$
Comparing \eqref{eqn:tracepq} and \eqref{eqn:tracepp}, we now see that $L^q(\mu)=L^p(\mu_{\mathcal Z})$. From this we readily conclude, first, that $p=q$ (e.g., by invoking the classical fact that otherwise $\ell^p$ is not isomorphic to $\ell^q$, see \cite{AK}) and, secondly, that the corresponding weight sequences are equivalent in the sense that $a_n\asymp1-|z_n|$ (otherwise the associated weighted $\ell^p$ spaces, namely $L^p(\mu)$ and $L^p(\mu_{\mathcal Z})$, would be different). The proof is complete.

\section{Proofs of Theorem \ref{thm:bergman} and Proposition \ref{prop:discon}}

\noindent{\it Proof of Theorem \ref{thm:bergman}.} Once again, we only need to prove that (i) implies (ii), and the strategy of the proof below will mimic that of the preceding one. We shall distinguish two cases.

\smallskip {\it Case 1:} Assume first that $\mu$ is supported on a uniqueness set for $A^p$, in the sense that every $A^p$ function vanishing $\mu$-a.e. is null. The kernel of the embedding operator $J:A^p\hookrightarrow L^q(\mu)$ is then trivial (we now write $J$ for $J_{p,q}(\mu)$), while its range is closed in $L^q(\mu)$ by hypothesis. Consequently, by the open mapping theorem, $J$ maps $A^p$ isomorphically onto $J(A^p)$, so that
\begin{equation}\label{eqn:isompqberg}
c\|f\|_{A^p}\le\|f\|_{L^q(\mu)}\le C\|f\|_{A^p},\qquad f\in A^p,
\end{equation}
for some constants $C,c>0$. (As before, we recall that the open mapping theorem applies to $F$-spaces, not just to Banach spaces; see \cite[Chapter 2]{R}. This is relevant when dealing with the quasi-Banach situation where $p$ or $q$ is less than $1$.)

\par We now claim that \eqref{eqn:isompqberg} can only hold if $p=q$. To see why, choose a positive integer $N$ such that the numbers $\widetilde p:=Np$ and $\widetilde q:=Nq$ satisfy $\min(\widetilde p,\widetilde q)>2$. For any $g\in A^{\widetilde p}$ we have $g^N\in A^p$, and applying \eqref{eqn:isompqberg} with $f=g^N$ yields
\begin{equation}\label{eqn:isompqbergtilde}
c^{1/N}\|g\|_{A^{\widetilde p}}\le\|g\|_{L^{\widetilde q}(\mu)}\le C^{1/N}\|g\|_{A^{\widetilde p}},\qquad g\in A^{\widetilde p}.
\end{equation}
This means that $A^{\widetilde p}$ embeds into $L^{\widetilde q}(\mu)$ and is isomorphic to its image (under the embedding map that arises), this image being a closed subspace of $L^{\widetilde q}(\mu)$. Since $A^{\widetilde p}$ is also isomorphic to $\ell^{\widetilde p}$ (see \cite[p.\,89]{W}), it follows that $L^{\widetilde q}(\mu)$ contains an isomorphic copy of $\ell^{\widetilde p}$. By virtue of Lemma \ref{lem:isomembed} (used with $\widetilde p$ and $\widetilde q$ in place of $p$ and $q$), this is only possible if $\widetilde p=\widetilde q$, or equivalently, if $p=q$. 

\par Finally, letting $q=p$ in \eqref{eqn:isompqberg}, we obtain 
\begin{equation}\label{eqn:ppberg}
c\|f\|_{A^p}\le\|f\|_{L^p(\mu)}\le C\|f\|_{A^p},\qquad f\in A^p,
\end{equation}
which means that $\mu$ is a sampling measure for $A^p$.

\smallskip {\it Case 2:} Now assume that $\mu$ is supported on an infinite set, say $\mathcal Z$, which is a zero-set for $A^p$ (meaning that there is a non-null $A^p$ function vanishing on $\mathcal Z$). If $\mathcal Z=\{z_n\}$, then $\mu$ takes the form \eqref{eqn:musum} for some positive numbers $a_n$ satisfying $\sum_na_n<\infty$. We may thus view $L^q(\mu)$ as a weighted $\ell^q$-space, the norm of a sequence $\mathcal W=\{w_n\}$ therein being defined by \eqref{eqn:sumanwnq}.

\par We further observe, exactly as before, that $\mathcal F$ (the set of sequences with finitely many nonzero elements) is dense in $L^q(\mu)$; we also claim that $\mathcal F\subset J(A^p)$. To verify this last inclusion, consider an arbitrary sequence 
$$\mathcal W=(w_1,\dots,w_N,0,0,\dots)$$
from $\mathcal F$, and let $\Phi_N$ be an $A^p$ function that vanishes on $\{z_n\}_{n>N}$ but takes nonzero values at $z_1,\dots,z_N$. (To arrive at such a function, take a nontrivial $\Phi\in A^p$ with $\Phi|_{\mathcal Z}=0$ and divide out the unwanted zeros at the $z_n$'s with $1\le n\le N$.) Finally, letting $Q$ be a polynomial for which
$$Q(z_n)=\f{w_n}{\Phi_N(z_n)}\quad\text{\rm for}\quad n=1,\dots,N,$$
and putting $h=Q\Phi_N$, we see that $h\in A^p$ and $h|_{\mathcal Z}=\mathcal W$. Now, since $J(A^p)$ contains a set (namely, $\mathcal F$) which is dense in $L^q(\mu)$, we deduce from condition (i) that $J(A^p)=L^q(\mu)$, or equivalently, 
\begin{equation}\label{eqn:bergtracepq}
A^p|_{\mathcal Z}=L^q(\mu).
\end{equation}

\par Because $L^q(\mu)$ is an ideal space, we may couple \eqref{eqn:bergtracepq} with Lemma \ref{lem:idealtrsp}, part (b), to infer that $\mathcal Z$ is an $A^p$-interpolating sequence. This means that 
\begin{equation}\label{eqn:bergtracepp}
A^p|_{\mathcal Z}=L^p(\nu_{\mathcal Z}), 
\end{equation}
where 
\begin{equation}\label{eqn:nuzdef}
\nu_{\mathcal Z}:=\sum_n(1-|z_n|)^2\de_{z_n}.
\end{equation}
Comparing \eqref{eqn:bergtracepq} and \eqref{eqn:bergtracepp} yields $L^q(\mu)=L^p(\nu_{\mathcal Z})$. From this we deduce, exactly as we did at the end of Section 3, that $p=q$ and finally conclude that $a_n\asymp(1-|z_n|)^2$ (otherwise the corresponding weighted $\ell^p$ spaces, namely $L^p(\mu)$ and $L^p(\nu_{\mathcal Z})$, would not be the same). The proof is complete.
\qed

\bigskip 

\noindent{\it Proof of Proposition \ref{prop:discon}.} For a fixed $p_0>0$, put $b:=2^{1/p_0}$ and let $\mathcal Z=\{z_n\}$ be the zero-set of the so-called Horowitz product $H_b$ given by 
$$H_b(z):=\prod_{k=1}^\infty\left(1-bz^{2^k}\right),\qquad z\in\D.$$
(Such products were introduced in \cite{Ho}.) Next, consider the measure $\mu:=\nu_{\mathcal Z}$, where $\nu_{\mathcal Z}$ is defined in terms of $\mathcal Z$ as in \eqref{eqn:nuzdef}. It is known that $\mathcal Z$ is an $A^p$-interpolating sequence if and only if $0<p<p_0$, while $\mathcal Z$ is a sampling sequence for $A^p$ (i.e., $\mu$ is a sampling measure for $A^p$) if and only if $p>p_0$; these facts are established in \cite[Section 6.7]{DS}. It follows that our current measure $\mu$ satisfies condition (ii) of Theorem \ref{thm:bergman} precisely when $p\ne p_0$. The theorem tells us therefore that the set $\mathcal C_B(\mu)$ coincides with $(0,\infty)\setminus\{p_0\}$, and we are done.
\qed





\end{document}